\documentclass[reqno,a4paper]{amsart}
\usepackage{latexsym,amssymb}
\usepackage[ansinew]{inputenc}
\usepackage[T1]{fontenc}
\usepackage{times,mathptmx}
\usepackage{color}
\usepackage{hyperref}
\theoremstyle{plain}
\newtheorem{PropIntro}{Proposition}
\newtheorem{ThmIntro}[PropIntro]{Theorem}

\newtheorem{RemIntro}[PropIntro]{Remark}

\newtheorem{Prop}{Proposition}[section]
\newtheorem{Thm}[Prop]{Theorem}
\newtheorem{Lem}[Prop]{Lemma}

\newtheorem{Rem}[Prop]{Remark}

\newcommand{\IN}{\mathbb{N}}

\newcommand{\IR}{\mathbb{R}}

\newcommand{\C}{\mathcal{C}}

\newcommand{\N}{{{\mathcal N}}}

\newcommand{\E}{{{\mathcal E}}}

\newcommand{\abs}[1]{\mathopen\vert#1\mathclose\vert}

\newcommand{\norm}[1]{\mathopen\Vert#1\mathclose\Vert}
\newcommand{\intd}{\,{\mathrm d}}
\renewcommand{\phi}{\varphi}
\renewcommand{\epsilon}{\varepsilon}

\renewcommand{\ge}{\geqslant}
\renewcommand{\subset}{\subseteq}

\newcommand{\rr}{\mathbb{R}}

\definecolor{umhblue}{rgb}{.69,.75,.86}
\definecolor{umhdarkblue}{rgb}{.15,.15,.53}% for ``Darmstadt'' theme
\definecolor{umhred}{rgb}{0.4,.0,.0}
\definecolor{mygreen}{RGB}{0, 165, 0}
\definecolor{definition}{rgb}{0.45, 0.61, 0.96}
\definecolor{problem}{rgb}{0.84,0.5,0}

\begin{document}

\title[Estimates and asymptotic behavior for Lane Emden problem]{Lane Emden
problems: asymptotic behavior of low energy nodal solutions}
\author[M.~Grossi, C.~Grumiau, F.~Pacella]{Massimo Grossi, Christopher
  Grumiau, Filomena Pacella}

\address{
Dipartimento di Matematica\\
   Universita' di Roma  ``Sapienza''\\
  P.~le A.~Moro 2, 00185 Roma, Italy}
\email{grossi@mat.uniroma.it (M.~Grossi)}
\email{pacella@mat.uniroma.it (F.~Pacella)}

\address{
  Institut de Math{\'e}matique\\
  Universit{\'e} de Mons\\
  20, Place du Parc, B-7000 Mons, Belgium}
\email{Christopher.Grumiau@umons.ac.be (C.~Grumiau)}

\begin{abstract}
We study the  nodal solutions of the Lane Emden Dirichlet problem
\begin{equation*}
\left\{
\begin{aligned}
-\Delta u&= \abs{ u}^{p-1}u,&&\text{ in } \Omega, \\
u&=0,&&\text{ on } \partial \Omega,
\end{aligned}
\right.
\end{equation*}
where $\Omega$ is a smooth bounded domain in  $\IR^2$ and $p>1$.
We consider solutions $u_p$ satisfying $$p \int_{\Omega}
\abs{\nabla
  u_p}^2\to 16\pi e\quad\hbox{ as }p\rightarrow+\infty\qquad (*)$$
and we are interested in the shape and the
asymptotic behavior as $p\rightarrow+\infty$.

First we prove that (*) holds for least energy nodal solutions.
Then we obtain some estimates and the asymptotic profile of this
kind of solutions. Finally, in some cases, we prove that $pu_p$
can be characterized as the difference of two Green's functions
and the nodal line intersects the boundary of $\Omega$, for large
$p$.
\end{abstract}



\thanks{This work has been done while the second author was visiting
  the Mathematics Department of the University of Roma ``
  Sapienza'' supported by INDAM-GNAMPA.  He also acknowledges the national bank of Belgium and the program
``Qualitative study of
 solutions of variational elliptic partial differerential equations. Symmetries,
bifurcations, singularities, multiplicity and numerics'' of the FNRS, project
2.4.550.10.F of the Fonds de la Recherche Fondamentale Collective for
the partial support.}

\maketitle

%%%%%%%%%%%%%%%%%%%%%%%%%%%%%%%%%%%%%%%%%%%%%%%%%%%%%%%
%%%%%%%%%%%%%%%%%%%%%%%%%%%%%%%%%%%%%%%%%%%%%%%%%%%%%%%
\section{Introduction} \label{Section-Intro}%%%%%%%%%%%
%%%%%%%%%%%%%%%%%%%%%%%%%%%%%%%%%%%%%%%%%%%%%%%%%%%%%%%
%%%%%%%%%%%%%%%%%%%%%%%%%%%%%%%%%%%%%%%%%%%%%%%%%%%%%%%

We consider the superlinear elliptic boundary value problem
\begin{equation*}
\tag{\protect{$\mathcal{P}_p$}}
\label{pblP}
\left\{
\begin{aligned}
-\Delta u&= \abs{ u}^{p-1}u,&&\text{ in }\Omega, \\
u&=0,&&\text{ on } \partial\Omega,
\end{aligned}
\right.
\end{equation*}
where $\Omega$ is a smooth bounded domain in $\IR^{2}$ and
$p>1$.\\
By standard variational methods we know that problem~\eqref{pblP}
has a positive ground state solution. Moreover many other results
about the multiplicity and the qualitative properties of positive
solutions in various types of domains have been obtained in the
last decades.

In this paper we are interested in studying sign changing
solutions of~\eqref{pblP}. In contrast with the case of positive
solutions not much is known on nodal solutions of~\eqref{pblP}, in
particular about their qualitative behavior. Let us therefore
recall some recent results. In the paper~\cite{ccn} A.~Castro,
J.~Cossio and J.~M.~Neuberger proved the existence of a nodal
solution with least energy among nodal solutions, which is
therefore referred to as the \emph{least energy nodal solution} of
Problem \eqref{pblP}. T.~Bartsch and T.~Weth showed that these
solutions possess exactly two nodal regions and have Morse index
two (see \cite{bartweth}). Since positive ground state solutions
have the symmetries of the domain $\Omega$, if $\Omega$ is convex,
by the classical result of \cite{gnn} , a natural question is
whether least energy nodal solutions also inherit the symmetries
of the domain $\Omega$. In \cite{aftalion} A.~Aftalion and
F.~Pacella proved that, in a ball or in a annulus, a least energy
nodal solution cannot be radial. In fact, in dimension $N$, they
cannot be even with respect to more than $N-1$ orthogonal
directions. They also proved that the nodal set touches the
boundary. On the other hand, T.~Bartsch, T.~Weth and M.~Willem in
\cite{wil1} and F.~Pacella and T.~Weth in \cite{pw}, with
different methods, obtained partial symmetry results: they showed
that on a radial domain, a least energy nodal solution $u$ has the
so-called foliated Schwarz symmetry, i.e.\ $u$ can be written as
$u(x)=\Tilde{u}(\abs{x},\xi \cdot x)$, where $\xi \in \IR^N$ and
$\Tilde{u}(r,\cdot)$ is nondecreasing for every $r>0$. In fact, as
they are not radial, $\Tilde{u}(r,\cdot)$ is increasing. In
dimension $N$, it implies that the least energy nodal solutions
are even with respect to $N-1$ orthogonal directions. Concerning
the ``last direction'', in \cite{bbgv,gt}, D.~Bonheure,
V.~Bouchez, C.~Grumiau, C.~Troestler and J.~Van~Schaftingen proved
that for $p$ close to $1$ the least energy nodal solution must be
odd with respect to this direction. Moreover, it is unique up to a
rotation. For general open bounded domains, they prove that least
energy nodal solutions must respect the symmetries of their
orthogonal projection on the second eigenspace of $-\Delta$ when
$p$ is close to $1$.

In this paper we study the profile and other qualitative
properties of low energy nodal solutions of problem~\eqref{pblP}
as $p\to +\infty$ and $\Omega\subset \IR^2$ is any bounded smooth
domain. For ground state positive solutions the same analysis has
been done by  X.~Ren and J.~Wei in~\cite{ren} and ~\cite{rw},
obtaining, in particular, $L^\infty$ estimates. This result has
been improved by Adimurthi and M. Grossi in~\cite{grossi} (see
also ~\cite{elgrossi}) who computed the exact value of the
$L^\infty$-norm at the limit, by a different approach.

Here by low energy we mean that we are interested in the families
of nodal solutions $(u_p)_{p>1}$ satisfying
\begin{equation*}
\tag{A}
\label{assum}
p\int_{\Omega}\abs{\nabla
  u_p}^2 \to 16\pi e\quad\hbox{ as }p\rightarrow+\infty.
\end{equation*}
Note that as a consequence
of (\cite{ren}) and as it will be clear later, this kind of solutions
cannot have more than $2$ nodal regions for $p$ large.\\
Let us observe that there are nodal solutions of ~\eqref{pblP}
satisfying ~\eqref{assum}. In fact least energy nodal
solutions are among those and we have:

\begin{ThmIntro}
\label{intro1}
The condition~\eqref{assum}  holds for any family of least energy nodal
solutions.
\end{ThmIntro}
To describe our results we need some notations. In
$H^{1}_{0}(\Omega)$, we use the scalar product
$(u,v)=\int_{\Omega}\nabla u\cdot\nabla v$ and denote by
$\norm{\cdot }_{q}$ the usual norm in $L^q(\Omega)$ and by $d(x,
D)$ the distance between a point $x\in \IR^2$ and the set
$D\subset \IR^2$.
 Let us consider a family of nodal solutions
$(u_p)_{p>1}$.  Throughout the paper, we assume that $u_p$ are low
energy solutions, i.e.~\eqref{assum} holds. The positive part
$u_p^+$ (resp.\ negative part $u_p^-$) are defined as $u_p^+ :=
\max (u_p,0)$ (resp.\ $u_p^-:=\min(u_p,0)$).

Let us define the families $(x_p^+)_{p>1}$ (resp.\ $(x_p^-)_{p>1}$) of
maximum (resp.\
minimum) points in $\Omega$ of $u_p$, i.e.\
$u_p(x_p^+)=\norm{u_p^+}_{\infty}$ and $u_p(x_p^-)=
-\norm{u_p^-}_{\infty}$ and assume w.l.o.g.\ that  $u_p(x_p^+)=
\norm{u_p}_{\infty}$,  i.e.\ $u_p(x_p^+) \geq -u_p(x_p^-)$.  To start with,
we prove that $x_p^+$ cannot go ``too
fast'' to the boundary of $\Omega$ which is the key point to make
some rescaling around $x_p^+$ and obtain a limit profile on
$\IR^2$. More precisely we
prove that $\frac{d(x_{p}^+,\partial\Omega)}{\varepsilon_{p}}\to
+\infty$ (see Proposition~\ref{away...}), where
\begin{equation*}
\varepsilon_p^{-2}:=pu_p(x_p^+)^{p-1}.
\end{equation*}
Then we get the following result.
\begin{ThmIntro}
\label{intro2}
The scaling of $u_p$ around $x_p^+$:
\begin{equation*}
z_p(x):= \frac{p}{u_p(x_p^+)}(u_p(\varepsilon_p x+x_p^+)-u_p(x_p^+))
\end{equation*}
defined on $\Omega^{+}(\varepsilon_p):=\frac{\Omega -
x_p^+}{\varepsilon_p}$ converges, as $p\rightarrow\infty$ to a
function $z$ in $C^2_{\text{loc}}(\IR^2)$. Moreover $z$ must solve
the equation $-\Delta z = e^z$ on $\IR^2$, $z\leq 0$, $z(0)=0$,
$\int_{\IR^2} e^z=8\pi$ and
$z(x)=\log\left(\frac{1}{(1+\frac{1}{8}\abs{x}^2)^2}\right)$.
\end{ThmIntro}

As a consequence of the previous theorem, we deduce that
$\varepsilon_p^{-1}d(x_p^{+}, NL_p) \to +\infty$ as
$p\rightarrow\infty$, where $NL_p$ denotes the nodal line of
$u_p$. So, in some sense, the rescaled solution about $x_p^+$
ignores  the other nodal domain of $u_p$.  This implies that we
can repeat the same kind of rescaling argument in the positive
nodal domain $\Tilde{\Omega}_{p}^+:=\{x\in\Omega : u_p(x)>0\}$ of
$u_p$. Hence, defining $\Tilde{\Omega}^+(\varepsilon_{p}):=
\frac{\Tilde{\Omega}_{p}^+-   x_p^+}{\varepsilon_{p}}$, we get the
analogous of Theorem~\ref{intro2}:

\begin{ThmIntro}
\label{intro3} The function
$z_p:\Tilde{\Omega}^+(\varepsilon_{p})\to \IR$ converges, as
$p\rightarrow+\infty$, to a function $z$ in
$C^2_{\text{loc}}(\IR^2)$ as $p\rightarrow\infty$. Moreover $z$
must solve the equation $-\Delta z = e^z$ on $\IR^2$, $z\leq 0$,
$z(0)=0$, $\int_{\IR^2} e^z= 8\pi$ and
$z(x)=\log\left(\frac{1}{(1+\frac{1}{8}\abs{x}^2)^2}\right)$.
\end{ThmIntro}

At this point, to the aim of studying the negative part $u_p^-$, let
us observe that we can have
two types of families of solutions satisfying the
assumption~\eqref{assum}, the ones which satisfy
\begin{enumerate}
\item[$(B)$] there exists $K\ge0$ such that  $p\left(u_p(x_p^+) +
    u_p(x_p^-)\right)\rightarrow K$;
\end{enumerate}
and the ones which satisfy
\begin{enumerate}
\item[$(B')$]   $p (u_p(x_p^+) + u_p(x_p^-))\rightarrow\infty$.
\end{enumerate}
The meaning of $(B)$ is that the speeds of convergence of the maximum
and the minimum of $u_p$ (multiplied by $p$) are
comparable. Instead the condition $(B')$ implies that one of the two
values converges faster than the other one.
\begin{RemIntro}
\label{R} It is easy to see that nodal solutions of type $(B)$
exist. Indeed, if $\Omega$ is a ball, it is enough to consider the
antisymmetric, with respect to a diameter, solution with two nodal
regions. We believe that also solution of type $(B')$ should exist
and we conjecture that the radial solution in the ball, with two
nodal regions, should be of type $(B')$. However, the complete
characterization of low energy solutions in the ball will be
analyzed in a subsequent paper.
\end{RemIntro}
In this paper we investigate the alternative $(B)$ that we
conjecture holding for the least energy nodal solutions.\\
First, we prove that, as for $x_p^+$, the condition $(B)$ implies that
$\varepsilon_{p}^{-1}d(x_p^-,\partial\Omega)\to +\infty$ as
$p\rightarrow\infty$. Then we get the following result.

\begin{ThmIntro}
\label{intro2bis}
If $(B)$ holds then the scaling of $u_p$ around $x_p^-$
\begin{equation*}
z_p^-(x):= \frac{p}{u_p(x_p^+)}(-u_p(\varepsilon_p x+x_p^-)-u_p(x_p^+))
\end{equation*}
defined on $\Omega^{-}(\varepsilon_p):=\frac{\Omega -
x_p^-}{\varepsilon_p}$ converges, as $p\rightarrow+\infty$, to a
function $z$ in $C^2_{\text{loc}}(\IR^2)$. Moreover $z$ must solve
the equation $-\Delta z = e^z$ on $\IR^2$, $z\leq 0$,
$\int_{\IR^2} e^z=8\pi$ and
$z(x)=\log\left(\frac{\mu}{(1+\frac{\mu}{8}\abs{x}^2)^2}\right)$
for some $0 < \mu\leq 1$.  When $K=0$ in condition~$(B)$, we get
$\mu =1$.
\end{ThmIntro}

As for the case of $x_p^+$, as a consequence of
Theorem~\ref{intro2bis}, we get that
$\varepsilon_{p}^{-1}d(x_{p}^-,NL_{p})\to +\infty$, which allows
to do the  same rescaling in the  negative nodal domain
$\Tilde{\Omega}_p^-:= \{x\in\Omega : u_p(x)<0\}$, obtaining the
analogous of Theorem~\ref{intro2bis}.

\begin{ThmIntro}
\label{intro4}
If $(B)$ holds, the function
\begin{equation*}
z_p^-(x):= \frac{p}{\norm{u_p}_{\infty}}\left( -u_p^-(\varepsilon_{p}
  x+x_p^-)-\norm{u_p}_{\infty}\right)
\end{equation*}
defined on $\Tilde{\Omega}^-(\varepsilon_{p}):= \frac{\Tilde{\Omega}_{p}^--
  x_p^-}{\varepsilon_{p}}$  converges, as
$p\rightarrow+\infty$, to a function $z$ in
$C^2_{\text{loc}}(\IR^2)$. Moreover $z$ must solve the equation
$-\Delta z = e^z$ on $\IR^2$, $z\leq 0$,  $\int_{\IR^2} e^z =
8\pi$ and
$z(x)=\log\left(\frac{\mu}{(1+\frac{\mu}{8}\abs{x}^2)^2}\right)$
for some $0 <\mu\leq 1$.  When $K=0$ in condition~$(B)$, we get
$\mu =1$.
\end{ThmIntro}

\begin{RemIntro}
\label{R1}
Another natural condition to make the rescaling in the negative
nodal domain without assuming condition $(B)$ could be to consider the parameter
$$\tilde\varepsilon_{p}^{-2}=p \abs{u_p^-(x_p^-)}^{p-1}$$
which is now just related
to the negative part of $u$ (we are not using the $L^\infty$-norm of
$u_p$ but the $L^\infty$-norm of $u_p^-$) and assume that
$\tilde\varepsilon_{p}^{-1}d(x_p^-,NL_p)\to
+\infty$ (as before $NL_p$ is the nodal line of $u_p$) . This
assumption is essentially equivalent to condition $(B)$ and allows to
prove that
$\tilde\varepsilon_{p}^{-1}d(x_p^-,\partial\Omega)\to+\infty$ (see
Proposition~\ref{away2...}). Then one could repeat
the proof of Theorem~\ref{intro4} obtaining for $z_p(x):=
\frac{p}{u_p(x_p^-)}\left( u_p^-(\tilde\varepsilon_{p}
  x+x_p^-)-u_p(x_p^-)\right)$
the same assertion as for $z_p^-$.
\end{RemIntro}
If the positive part of $u$, i.e. $u_p^+$, as a solution
of~\eqref{pblP} in $\Tilde{\Omega}^+(\varepsilon_{p})$, has Morse
index one then the previous results allow to obtain the exact
value of the limits of $\norm{u_p^\pm}_{\infty}$, as
$p\rightarrow+\infty$.
\begin{ThmIntro}
\label{intro5} Let us assume that the Morse index of $u_p^+$ as
a solution of \eqref{pblP} in $\Tilde{\Omega}_p^+$ is one. Then we
have: $\norm{u_p^+}_{\infty}\to e^{1/2}$.  If also $(B)$ holds
then $\norm{u_p^-}_{\infty}\to e^{1/2}$.
\end{ThmIntro}
The result of the previous statement is similar to the one
obtained in~\cite{grossi} for the least energy positive solution
of \eqref{pblP}.

Let us remark that the additional assumption on the Morse index of
$u^+_p$ holds for any nodal solutions with Morse index $2$, hence,
in particular, for least energy nodal solutions.

Our last result gives the asymptotic behavior of the nodal
solutions in the whole domain $\Omega$.

Let us denote by $G(x,y)=-\frac1{2\pi}\log|x-y|+H(x,y)$ the
Green's function of $\Omega$ and by $H$ its regular part. Finally,
let $x^\pm$ be the limit point of $x_{p}^{\pm}$  as $p\to
+\infty$.
\begin{ThmIntro}
\label{introthm} Under the same hypothesis of Theorem
\ref{intro5}, $p u_p$ converges, as $p\rightarrow+\infty$, to the
function
  $8\pi e^{1/2}( G(\cdot, x^+)-G(\cdot, x^-))$  in
  $\C^2_{\text{loc}}(\overline{\Omega}\setminus\{x^-, x^+\})$ and $x^+
  \neq x^-\in\Omega$. Moreover the limit points $x^+$ and $x^-$ satisfy the system

\begin{equation*}
\left\{
\begin{aligned}
\frac{\partial G}{\partial x_i}(x^+, x^-)-\frac{\partial H}{\partial
  x_i}(x^+, x^+)=0,\\
\frac{\partial G}{\partial x_i}(x^-, x^+)-\frac{\partial H}{\partial
  x_i}(x^-, x^-)=0,
\end{aligned}
\right.
\end{equation*}
\vskip0.21cm\noindent
for $i=1,2$. Finally, the nodal line of $u_p$
intersects the boundary of $\Omega$ for $p$ large.
\end{ThmIntro}
The result of Theorem~\ref{introthm} gives a very accurate description
of the profile of the low energy solutions of
type $(B)$ in terms of the Green function of $\Omega$ and of its
regular part. It is also remarkable that the property
that the nodal line intersects $\partial\Omega$ holds for this kind of
solutions in any bounded domain $\Omega$, extending so the result
proved in
\cite{aftalion} for least energy solutions in balls or annulus. It is
also reminiscent of the property of the second eigenfunction
of the laplacian in planar convex domains (see~\cite{melas}), though
we are not analyzing the case of $p$ close to $1$ as in
\cite{bbgv,gt}.\\
Let us remark that nodal solutions with this property have been
constructed in~\cite{espo1,espo2}.\\
Finally we would like to point out that our analysis is similar to
the one carried out in~\cite{bep3,bep1,bep2} for low energy nodal
solutions of an almost critical problem or of the Brezis-Nirenberg
problem in dimension $N\ge3$. However, the techniques and the
proofs are completely different since in~\cite{bep3,bep1,bep2} the
nodal solutions whose energy is close to $2S_N$ ($S_N$ is the best
Sobolev constant in $\rr^N$) can be written almost explicitly.

The outline of the paper is as follows. In Section~$2$, we recall
the variational characterization of the problem and  we prove
Theorem~\ref{intro1} and some useful asymptotic estimates. In
Section~$3$, we show that $x_p^+ $ cannot go too fast to the
boundary and then prove Theorem~\ref{intro2} and
Theorem~\ref{intro2bis} using a rescaling argument on the whole
domain $\Omega$. Then, using a rescaling argument on the nodal
domains, we prove Theorem~\ref{intro3} and Theorem~\ref{intro4}.
In Section~$4$, we improve the bounds given in Section~$2$ to
obtain Theorem~\ref{intro5}. Finally, in Section~$5$, we prove
Theorem~\ref{introthm}.

{\bf Acknowledgment} We would like to thank A. Adimurthi for some
useful discussions, in particular about the proof of Proposition
~\ref{away...}.

\section{Variational setting and estimates}\ \\
We recall that solutions of problem \eqref{pblP} are the critical points of the
energy functional $\E_{p}$ defined on $H^{1}_{0}(\Omega)$ by
\[\E_{p}(u)=\frac{1}{2}\int_{\Omega}\vert \nabla u\vert^{2}
-\frac{1}{p+1}\int_{\Omega}\vert u\vert^{p+1}.\]
The Nehari manifold $\mathcal{N}_{p}$ and the nodal Nehari set
$\mathcal{M}_{p}$ are defined by
\[\mathcal{N}_{p}:=\{u\in H^{1}_{0}(\Omega)\setminus\{0\}:\langle
\intd \E_{p}(u),u\rangle =0\},\quad \mathcal{M}_{p}:=\{u\in
H^{1}_{0}(\Omega): u^{\pm}\in \mathcal{N}_{p}\},\]
where $u^{+}(x):=\max(u(x),0)$ and $u^{-}(x):=\min(u(x),0)$. If
$u\in H^{1}_{0}(\Omega)$, $u^{+}\neq 0$ and $u^{-}\neq 0$ then
$u\in\mathcal{M}_{p}$ if and only if
\begin{equation}\label{5}
\int_{\Omega}\vert\nabla u^{+}\vert^{2}=\int_{\Omega}\vert
u^{+}\vert^{p+1} \text{ and } \int_{\Omega}\vert\nabla u^{-}\vert^{2}
= \int_{\Omega}\vert u^{-}\vert^{p+1}.
\end{equation}

For any $u\neq 0$ fixed, there exists a unique multiplicative factor
$\alpha$ such that $\alpha u\in \N_p$. If $u$ changes sign then there
exists an unique couple $(\alpha_+, \alpha_-)$ such  that $\alpha_+
u^+ + \alpha_- u^-\in \mathcal{M}_p$.

The interest of $\mathcal{N}_p$ (resp.\ $\mathcal{M}_{p}$) comes from
the fact that it
contains all the non-zero (resp.\ sign-changing) critical points of
$\E_{p}$. If $u$
minimizes $\E_{p}$ on $\mathcal{N}_p$ (resp.\ $\mathcal{M}_{p}$) then
$u$ is a (resp.\ nodal) solution
of Problem~\eqref{pblP} usually referred to as the \emph{ground state
  solutions} (resp.\ \emph{least energy nodal
solutions}). So, we need to solve
\begin{equation*}
\inf \left\{\left(
\frac{1}{2}-\frac{1}{p+1}\right)\int_{\Omega}\abs{\nabla u}^2\right\} \text{ on }
\int_{\Omega}\abs{\nabla u^{\pm}}^2=\int_{\Omega} (u^{\pm})^{p+1}
\end{equation*}
to characterize the least energy nodal solutions.
\begin{Thm}[T.~Bartsch, T.~Weth~\cite{bartweth}]
There exists a least energy nodal solution of problem \eqref{pblP}
which has exactly two nodal domains and Morse index $2$.
\end{Thm}

To start with, we show that each family of least energy nodal
solutions for  Problem~\eqref{pblP} is a family of low energy
nodal solutions, i.e. satisfies condition ($A$) of the
introduction.  To this aim let us prove an upper bound and a
control on the energy.
\begin{Lem}
\label{estH+} Let $(u_p)_{p>1}$ be a family of least energy nodal
solutions of Problem~\eqref{pblP}.  For any $\varepsilon
>0$, there exists $p_\varepsilon$ such that, for any $p\geq
p_\varepsilon$,
\begin{equation*}
p\E_p(u_p)= p\big( \frac{1}{2}-\frac{1}{p+1}\big)
\int_{\Omega}\abs{\nabla u_p}^2 \leq 8\pi e + \varepsilon.
\end{equation*}
\end{Lem}
\begin{proof}
Let $a,b \in \Omega$.  Let us
consider $0<r<1$ such that $B(a,r), B(b,r)\subset \Omega$ and
$B(a,r)\cap B(b,r)=\emptyset$.  Then, we define a cut-off function
$\phi :\Omega\to
[0,1]$ in $\C^\infty_0(\Omega)$ such that
\begin{equation*}
\phi(x):=\left\{
\begin{aligned}
& 1 &&\text{ if } \abs{x-a}<r/2, \\
&0 &&\text{ if } \abs{x-a}\geq r.
\end{aligned}
\right.
\end{equation*}

First we introduce the family of functions $\Bar{W}_p: \Omega \to
\IR$ which are defined on $B(a,r)$ as
\begin{equation*}
\Bar{W}_p(x):= \phi(x)\sqrt e \left( 1 +
  \frac{z\big(\frac{x-a}{\varepsilon_p} \big)}{p}\right)
\end{equation*}
where $z(x) = -2\log(1+ \frac{\abs{x}^2}{8})$ and
$\varepsilon_p^2:=\frac{1}{p\sqrt{e}^{p-1}}$.  The functions
$\Bar{W}_p$ vanishes outside the ball $B(a,r)$.  We claim that
$$\int_{\Omega}\abs{\Bar{W}_p}^{p+1}=\frac{8\pi e}{p}+ o(1/p),$$
$$\int_{\Omega}\abs{\nabla\Bar{W}_p}^2=\frac{8\pi e}{p}+ o(1/p).$$

Indeed, setting $\frac{x-a}{\varepsilon_p}=\psi$ and using the
fact that $\int_{\IR^2}e^z = 8\pi$,
\begin{equation*}
\begin{split}
\int_{\Omega}\abs{\Bar{W}_p}^{p+1} &=
(\sqrt{e})^{p+1}\varepsilon_p^2\int_{\frac{\Omega - a}{\varepsilon_p}}
\phi(\varepsilon \psi + a )^{p+1} \left( 1+
  \frac{z(\psi)}{p}\right)^{p+1}\intd \psi \\
&= \frac{e}{p}\left( \int_{\IR^2}e^z + o(1) \right)\\
&= \frac{8\pi e}{p} + o(1/p).
\end{split}
\end{equation*}
Concerning  $\int_{\Omega}\abs{\nabla  \Bar{W}_p}^2$, we get that
\begin{equation*}
\begin{split}
\int_{\Omega}\abs{\nabla  \Bar{W}_p}^2 =& \int_{\Omega}  \phi^2(x)
\left\vert\nabla \left(\sqrt e (1+ \frac{z((x-a)/\varepsilon_p)}{p})
    \right)\right\vert^2 + \int_{\Omega}  \abs{\nabla\phi(x)}^2
 \left(\sqrt e (1+ \frac{z((x-a)/\varepsilon_p)}{p})\right)^2\\
& + 2 \int_{\Omega} \phi(x)\sqrt e (1+
\frac{z((x-a)/\varepsilon_p)}{p}) \nabla\phi(x)\cdot \nabla \big(\sqrt
e (1+ \frac{z((x-a)/\varepsilon_p)}{p})\big) .
\end{split}
\end{equation*}
The first term gives
\begin{equation*}
\begin{split}
\int_{\Omega}  \phi^2(x)
\left\vert\nabla \left(\sqrt e (1+ \frac{z((x-a)/\varepsilon_p)}{p})
    \right)\right\vert^2 &=  \frac{e}{p^2} \int_{\Omega} 16
  \phi^2(x)\frac{\abs{x-a}^2}{ (8\varepsilon_p^2 + \abs{x-a}^2)^2} \\
&=   \frac{16e}{p^2} \left\{\int_{B(a,r/2)}
  \frac{\abs{x-a}^2}{ (8\varepsilon_p^2 +
    \abs{x-a}^2)^2}+ \int_{\Omega \setminus B(a,r/2)}
  \phi^2(x)\frac{\abs{x-a}^2}{ (8\varepsilon_p^2 +
    \abs{x-a}^2)^2}\right\}\\
&=  \frac{16e}{p^2}  (2\pi \int_{0}^{r/2} \frac{\psi^3}{(8\varepsilon_p^2 +
\psi^2)^2} + O(1)).
\end{split}
\end{equation*}
Setting $\psi^2 =t$ and integrating, we get
\begin{equation*}
\int_{0}^{r/2} \frac{\psi^3}{(8\varepsilon_p^2 +
\psi^2)^2} = \frac{1}{2}
\log\left\vert\frac{\frac{r}{2}+8\varepsilon_p^2}{8\varepsilon_p^2}\right\vert
+ \frac{1}{2} \left( \frac{8\varepsilon_p^2}{8\varepsilon_p^2 +\frac{r}{2}}-1
\right) = -\log \abs{\varepsilon_p} + O(1).
\end{equation*}
So, we get
\begin{equation*}
\begin{split}
\int_{\Omega}  \phi^2(x)
\left\vert\nabla \left(\sqrt e (1+ \frac{z((x-a)/\varepsilon_p)}{p})
    \right)\right\vert^2  &= -\frac{32\pi e}{p^2} \left(
  \log \varepsilon_p + O(1)\right) \\
&= \frac{-32e \pi}{p^2} \left(-\frac{p-1}{4} + o(p) + O(1)\right)\\
&= \frac{8\pi e}{p} + o(1/p).
\end{split}
\end{equation*}
The second term gives the existence of a constant $K>0$ such that
\begin{equation*}
\begin{split}
 \int_{\Omega}  \abs{\nabla\phi(x)}^2
 \left(\sqrt e (1+ \frac{z((x-a)/\varepsilon_p)}{p})\right)^2 &=
 \int_{B(a,r)\setminus B(a,r/2)}  \abs{\nabla\phi(x)}^2
 \left(\sqrt e (1+ \frac{z((x-a)/\varepsilon_p)}{p})\right)^2 \\
&\leq  K \frac{\big(1 + 2\max_{x\in\Omega\setminus B(a,1/p)}\left\vert\log \left(
    \frac{\abs{x-a}^2p}{8}\right)\right\vert  + K\big)^2}{p^2} \\
& = o(1/p).
\end{split}
\end{equation*}
The third term can be treated with similar techniques.
So, finally, we get
\begin{equation*}
\int_{\Omega}\abs{\nabla  \Bar{W}_p}^2 =  \frac{8\pi e}{p} + o(1/p)
\end{equation*}
which proves the claim.

Then, we define  the family of test functions $W_p: \Omega \to
\IR$ which are defined on $B(a,r)$ as $\Bar{W}_p$ and on $B(b,r)$ as the
odd reflection of $\Bar{W}_p$. The functions
$W_p$ vanishes outside the two balls $B(a,r)$ and $B(b,r)$.
So, $\norm{\nabla W_p^{\pm}}_2^2 = \frac{8\pi e}{p} + o(1/p)$ and
$\norm{W^{\pm}_p}_{p+1}^{p+1} = \frac{8\pi e}{p}+ o(1/p)$.
Clearly, the unique multiplicative factor $\alpha_p:=\alpha_p^+$ such that
$\alpha_p^+ W_p^+ \in \N_p$ equals the unique multiplicative factor
$\alpha_p^-$ such that $\alpha_p^- W_p^- \in \N_p$.   To characterize
it, we need to solve
\begin{equation*}
\alpha_p^2
\norm{\nabla W^{\pm}_p}_2^2 = \alpha_{p}^{p+1}\norm{W_p^\pm}_{p+1}^{p+1}.
\end{equation*}
It implies that
\begin{equation}
\label{eq2}
\alpha_p = \left( \frac{\int_{\Omega}\abs{\nabla
      W_p^\pm}^2}{\int_{\Omega}\abs{W_p^{\pm}}^{p+1}}\right)^{\frac{1}{p-1}}
\to 1.
\end{equation}

So, as $u_p$ is a minimum for the $H^1_0$-norm on $\mathcal{M}_p$ and
$\int_{\Omega} \abs{\nabla u_p}^2=\int_{\Omega} \abs{\nabla
  u_p^+}^2+\int_{\Omega} \abs{\nabla u_p^-}^2$,
we conclude that
\begin{equation*}
p\left( \frac{1}{2}-\frac{1}{p+1}\right)\norm{\nabla
u_{p}}_2^2 \leq p
\left(\frac{1}{2}-\frac{1}{p+1}\right)2 (\alpha_p)^2 \int_{\Omega}
\abs{\nabla{W^+_p}}^2.
\end{equation*}
As the right-hand side converges to $8\pi e$, we get the
assertion.
\end{proof}

\begin{Lem}
\label{estH-} Let $(u_{p})_{p>1}$ be a family of least energy
nodal solutions of Problem~\eqref{pblP}. For any $\varepsilon
>0$, there exists $p_\varepsilon$ such that, for any $p\geq
p_\varepsilon$,
\begin{equation*}
p\E_{p}(u_{p})= p\big( \frac{1}{2}-\frac{1}{p+1}\big)
\int_{\Omega}\abs{\nabla u_{p}}^2 \geq 8\pi e - \varepsilon.
\end{equation*}
\end{Lem}
\begin{proof}
To do this, we prove that for any sequence $p_n\rightarrow+\infty$
$\liminf_{n\to+\infty}p_n
\left(\frac{1}{2}-\frac{1}{p_n+1}\right)\int_{\Omega}\abs{\nabla
  u_{p_n}^\pm}^2\geq 4\pi e$. On one hand, $1 =
\frac{\int_{\Omega}(u_{p_n}^\pm)^{p_n+1}}{\int_{\Omega}\abs{\nabla{u_{p_n}^\pm}}^2}$.
On the other hand, in \cite{ren} (page 752), it is proved that, for any $t>1$,
$\norm{u}_{t}\leq D_t
t^{1/2}\norm{\nabla u}_{2}$ where $D_t\to (8\pi
e)^{-1/2}$ is
independent of $u$ in $H^1_0(\Omega)$.

So, we obtain
\begin{equation*}
1\leq D_{p_n+1}^{p_n+1} (p_n+1)^{\frac{p_n+1}{2}}
\left(\int_{\Omega}\abs{\nabla u_{p_n}^\pm}^2\right)^{\frac{p_n-1}{2}},
\end{equation*}
i.e.\ $\int_{\Omega}\abs{\nabla u_{p_n}^\pm}^2 \geq
D_{p_n+1}^{-2\frac{p_n+1}{p_n-1}}
(p_n+1)^{-\frac{p_n+1}{p_n-1}}$. Thus,
\begin{equation*}
\left(\frac{1}{2}-\frac{1}{p_n+1}\right)
(p_n+1)^{\frac{p_n+1}{p_n-1}}\int_{\Omega}\abs{\nabla u_{p_n}^\pm}^2
\geq \left( \frac{1}{2}-\frac{1}{p_n+1}\right) D_{p_n+1}^{-2\frac{p_n+1}{p_n-1}}.
\end{equation*}
As $\frac{p_n}{(p_n+1)^{\frac{p_n+1}{p_n-1}}}$ converges to $1$ and the
right-hand side converges to $4\pi e$, we get the assertion.
\end{proof}

\textit{Proof of Theorem~\ref{intro1}} : it follows from
Lemma~\ref{estH+} and
Lemma~\ref{estH-}.\begin{flushright}$\square$\end{flushright}

\begin{Rem}
\label{rem} The proof of Lemma~\ref{estH-} does not depend on the
fact that $u_{p_n}$ is a least energy nodal solution. Indeed, for
any $(u_p)_{p>1}$ verifying~\eqref{assum}, as $p\to +\infty$, we
get
\begin{itemize}
\item $p\big(\frac{1}{2}-\frac{1}{p}\big) \int_{\Omega}\abs{\nabla
    u_{p}^{\pm}}^2\to 4\pi e$, $p \int_{\Omega}\abs{\nabla
    u_{p}^{\pm}}^2\to 8\pi e$  and $p
\int_{\Omega}\abs{\nabla   u_{p}}^2\to 16\pi e$.
\item $\E_{p}(u_{p})\to 0$, $\int_{\Omega} \abs{\nabla
    u_{p}}^2\to 0$, $\int_{\Omega} \abs{\nabla
    u_{p}^-}^2\to 0$ and $\int_{\Omega} \abs{\nabla
    u_{p}^+}^2\to 0$.
\end{itemize}
Moreover the proof of Lemma~\ref{estH-} implies, as corollary,
that $u_{p}$ has  $2$ nodal domains for $p$ large.
\end{Rem}

From now on, throughout the paper, we consider a family
$(u_p)_{p>1}$ of nodal solutions for which~\eqref{assum} holds.
The following result shows an asymptotic lower bound for the
$L^{\infty}$-norms of $u_p^+$ and $u_p^-$. We denote by
$\lambda_1(D)$ the first eigenvalue of $-\Delta$ with Dirichlet
boundary conditions in a domain $D$ and by $x_p^{\pm}$ both the
maximum or the minimum point of $u_p$, as defined in the
introduction.
\begin{Prop}
\label{estinfl}
For any $p>1$ we have that
$\abs{u_p(x_p^\pm)}\ge\lambda_1^{\frac{1}{p-1}}$ where
$\lambda_1:=\lambda_1(\Omega)$.
\end{Prop}
\begin{proof}
Using Poincar\'e's inequality, we get
\begin{equation*}
\begin{split}
1 = \frac{\int_{\Omega}\abs{u_p^{\pm}}^{p+1}}{\int_{\Omega}\abs{\nabla
  u_{p}^{\pm}}^2} & \leq
\frac{\abs{u_p(x_p^{\pm})}^{p-1}\int_{\Omega}(u_p^\pm)^2}{\int_{\Omega}\abs{\nabla
u_{p}^{\pm}}^2} \\
&\leq \abs{u_p(x_p^\pm)}^{p-1} \lambda_{1}^{-1}(\Tilde{\Omega}_{p}^{\pm}),
\end{split}
\end{equation*}
where $\Tilde{\Omega}_{p}^{\pm}$ are the nodal domains of $u_p$.
As $\Tilde{\Omega}_{p}^{\pm}\subset \Omega$, we have
$\lambda_1(\Tilde{\Omega}_{p}^{\pm})\geq \lambda_1$ which ends the
proof.
\end{proof}
\begin{Rem}
\label{rem2}
We have
\begin{itemize}
\item For any $\varepsilon>0$, $\abs{u_p(x_p^{\pm})}\geq 1-\varepsilon$ for
$p$ large. In particular this holds for $\norm{u_p}_{L^{+\infty}}$.
\item By Remark~\ref{rem}, as
  $\abs{u_p(x_{p}^\pm)}^{p-1}$ is bounded from below,
  $\frac{\abs{u_p(x_p^{\pm})}^{p-1}}{\int_{\Omega}\abs{\nabla
        u_p^\pm}^2}$ and
    $\frac{\abs{u_p(x_p^{\pm})}^{p-1}}{\int_{\Omega}\abs{\nabla
        u_p}^2}$ converge to $+\infty$ when $p\to +\infty$.
\end{itemize}
\end{Rem}

The next result  gives a direct argument to prove that the
$L^\infty$-norms of $u_p^+$ and $u_p^-$ are bounded. It will be
improved in the next sections.

\begin{Prop}
\label{linfb} We have that $u_{p}(x_{p}^{\pm})$ is bounded as
$p\to +\infty$.
\end{Prop}

\begin{proof}
Let us make the proof for the positive case. By
Proposition~\ref{estinfl}, we only have to prove  that
$u_{p}(x_{p}^{+})$ is bounded from above. Let us denote by $G$ the
Green's function on $\Omega$. As $\abs{G(x,y)}\leq C\abs{\log
\abs{x-y}}$ for any $x,y\in\Omega$ and some independent constant
$C>0$, using the Hölder inequality we have
\begin{equation*}
\begin{split}
u_{p}(x_{p}^+) &= \int_{\Omega}G(x_{p}^+,y)
\abs{u_{p}(y)^{p-1}}u_{p}(y)\intd y\\
&\leq C\int_{\Omega}\left\vert\log \abs{x_{p}^+ -y}\right\vert
\abs{u_{p}(y)^{p}}\intd y\\
&\leq C \left( \int_{\Omega}\left\vert\log\abs{
    x_{p}^+-y}\right\vert^{p+1}\intd y\right)^{\frac{1}{p+1}}
\left( \int_{\Omega}\abs{u_{p}}^{p+1}\right)^{\frac{p}{p+1}}.
\end{split}
\end{equation*}
Since $p \int_{\Omega}\abs{u_{p}}^{p+1}\to 16\pi e$ as
$p\rightarrow+\infty$ (see Remark~\ref{rem}), it is enough to show
the existence of a constant $C>0$ such that
\begin{equation*}
\int_{\Omega}\abs{\log \abs{x_{p}^+-y}}^{p+1}\intd y\leq C
(p+1)^{p+2}.
\end{equation*}
Let us consider $R>0$ such that $\Omega \subset B(x_{p},R)$ for
all $n$. Then there exists a constant $K>0$ such that
\begin{equation*}
\int_{\Omega}\abs{\log \abs{x_{p}^+-y}}^{p+1}\intd y\leq
\int_{B(x_{p},R)}\abs{\log \abs{x_{p}^+-y}}^{p+1}\intd y = K
\int_{0}^R \abs{\log r}^{p+1}r\intd r.
\end{equation*}
Integrating $([p]+1)$-times by parts, we get
\begin{equation*}
\begin{split}
\int_{\Omega}\abs{\log \abs{x_{p}^+-y}}^{p+1}\intd y&\leq
K\big\{\abs{\log(R)}^{p+1} + (p+1)\abs{\log(R)}^{p}+\cdots+
(p+1)\cdots(p-[p]+ 2)\\ & \abs{\log(R)}^{p -[p]+1}\big\} + K
(p+1)p\cdots (p- [p]+1) \int_{0}^R \abs{\log r}^{p -[p]} r\intd r.
\end{split}
\end{equation*}
Thus,  there exists $C$ such that for large $n$
\begin{equation*}
\int_{\Omega}\abs{\log \abs{x_{p}^+-y}}^{p+1}\intd y\leq C
(p+1)^{p+2},
\end{equation*}
which ends the proof.
\end{proof}

\section{Asymptotic behavior}

For the rest of the paper, w.l.o.g., let us assume that
$\norm{u_{p}}_{\infty}= u_p(x_p^+)$ for any $p>1$.

In this section we use several rescaling arguments to characterize the
asymptotic behavior of $u_p^\pm$.

Let us define $\varepsilon_p^2 :=\frac{1}{p
  u_p(x_p^+)^{p-1}}\to 0$ by Remark~\ref{rem2}.

\subsection{Control close to the boundary}

We prove that $x_p^+$  cannot go to the boundary of $\Omega$ too fast.

\begin{Prop}
\label{away...}
We have
\begin{equation}
\label{eq10} \frac{d(x_p^{+}, \partial
  \Omega)}{\varepsilon_p}\to +\infty
\end{equation}
as  $p\to +\infty$.
\end{Prop}

\begin{proof}
Let us argue by contradiction and assume that, for a sequence
$p_n\rightarrow+\infty$, $\frac{d(x_{p_n}^{+},
\partial  \Omega)}{\varepsilon_{p_n}}\to l\geq 0$ and that
$x_{p_n}^+\to x_* \in\partial \Omega$ (i.e.\ $\frac{d(x_{p_n}^{+},
x_*)}{\varepsilon_{p_n}}\to l$).

First, we treat the case when $\partial\Omega$ is flat
around $x_*$. We consider a semi-ball $D$ centered in $x_*$ with
radius $R$ such that $D\subset \Omega$ and the diameter of $D$
belongs to  $\partial\Omega$.  For large $n$, let us remark that
$x_{p_n}^+$ belongs to $D$.  Then, on $A:=B(x_*,R)$, we
consider the function $u_{p_n}^*$ which is defined as $u_{p_n}$ on $D$
and as the odd reflection of $u_{p_n}$
on $A\setminus D$.  It is a solution of $-\Delta u = \abs{u}^{p_n-1}
u$ on $A$. For large $n$, we consider
\begin{equation}
\label{eq11}
z_{p_n}^*(x) :=\frac{p_n}{u_{p_n}^*(x_{p_n}^+)}(u_{p_n}^*(\varepsilon_{p_n}x +
x_{p_n}^+)- u_{p_n}^*(x_{p_n}^+))
\end{equation}
on $\Omega^*_{p_n}:= \frac{A-x_{p_n}^+}{\varepsilon_{p_n}}\to
\IR^2$. On $\Omega^*_{p_n}$, we get from~\eqref{eq11}

\begin{equation*}
\left\{
\begin{aligned}
-\Delta z_{p_n}^*&= \left\vert
  1+\frac{z_{p_n}^*}{p_n}\right\vert^{p_n-1}\left(1+\frac{z_{p_n}^*}{p_n}\right),
& \\
&\left\vert 1+\frac{z_{p_n}^*}{p_n}\right\vert\leq 1.&
\end{aligned}
\right.
\end{equation*}

Let us fix $R>0$. For 
large $n$, $B(0, R)\subset \Omega^*_{p_n}$ and
we consider the problem
\begin{equation*}
\left\{
\begin{aligned}
-\Delta w_{p_n}&= \left\vert
  1+\frac{z_{p_n}^*}{p_n}\right\vert^{p_n-1}\bigl(1+\frac{z_{p_n}^*}{p_n}\bigr),&&\text{  in } B(0,R), \\
w_{p_n}&=0,&&\text{ on } \partial B(0,R).
\end{aligned}
\right.
\end{equation*}

Since, by~\eqref{eq11},  $\left\vert 1+\frac{z_{p_n}^*}{p_n}\right\vert\leq 1$,
we have that $\abs{w_{p_n}}$ is uniformly bounded by a constant $C$
independent of
$n$ by the maximum principle and the regularity theory. Moreover, because
$z_{p_n}\leq 0$, we have that $\psi_{p_n}= z_{p_n}-w_{p_n}$ is an harmonic function
which is uniformly bounded above.  By Harnack's inequality,
$\psi_{p_n}$ is bounded in $L^{\infty}(B(0, R))$ or
tends to $-\infty$ on each compact set of $B(0,R)$. As $\psi_{p_n}(0)=
z_{p_n}(0)-w_{p_n}(0)\geq -C$, we get that $\psi_{p_n}$ and $z_{p_n}$
are uniformly bounded on each compact set of $B(0,R)$.

Since we are assuming that
$\frac{d(x_{p_n},x_*)}{\varepsilon_{p_n}}\to l$ we get that
$y_n:=\frac{x_* - x_{p_n}^+}{\varepsilon_{p_n}}\in B[0, l+1]$ for large $n$
and $z_{p_n}(y_n) =-p_n \to -\infty$ which is a contradiction.

Next, we treat the case when $\partial \Omega$ is not
locally flat around $x_*$ but is a $\C^1$-curve.
We  consider a
$\C^1$-domain $D$ which is the intersection of  a fixed
neighborhood of $x_*$ and $\Omega$. Let us define the
square $Q:=(-1,1)^2$, $Q^+:=(-1,1)\times (0,1)\subset Q$ and $S:=
(-1,1)\times \{0\}$.

 We consider the change of
variables $\phi : D \to Q^+$ and
$\phi(D\cap \partial\Omega)=S $ (see~\cite{brezis} to get that $\phi$
is well-defined and can be assumed to be
$\C^1(\Bar{D})$). Moreover $\phi^-1 \in \C^1(\Bar{Q^+})$.

We fix a positive function $\theta\in C^2$ such that $\theta\circ
\phi^{-1} : \Bar{Q^+}\to\IR$ equals $0$ on $\partial Q^+ \setminus S$ and
$\partial_\nu\theta\circ \phi^{-1} =0$ on $S$ where $\partial_\nu$
denotes the normal derivative.
We extend $\theta\circ \phi^{-1}$ on $Q$ by
even symmetry with respect to $S$.

On $Q$, we define $\Tilde{u}_{p_n}$ as $\theta(\phi^{-1}(\cdot))
u_{p_n}(\phi^{-1}(\cdot))$ on $Q^+$ and the odd
symmetric function on $Q\setminus Q^+$ . Since $\theta u_{p_n}$ solves
\begin{equation}
\label{eq111}
-\Delta u = \theta  \abs{u_{p_n}}^{p_n-1}u_{p_n} -2
\nabla \theta \nabla u_{p_n} - (\Delta \theta)u_{p_n} =: g_{p_n}
\end{equation}
 with Dirichlet
boundary conditions on $D$, by the change of variables $y = \phi(x)$, we get
that  $\Tilde{u}_{p_n}$ solves  for some matrix $A_{p_n}$
\begin{equation*}
-div (A_{p_n}\nabla u) = h_{p_n}
\end{equation*}
with Dirichlet boundary conditions on
$Q$  and where $h_{p_n}$ is $g_{p_n}\circ \phi^{-1}$ on $Q^+$ and the
antisymmetric on $Q\setminus Q^+$.
Coming back to $\Omega$ by the change of variables $x=
\phi^{-1}(y)$ we get that $\theta u^*_{p_n}=
\Tilde{u}_{p_n}(\phi(\cdot))$ solves $-\Delta u = h_{p_n}\circ \phi$ on
$A:= \phi^{-1}(Q)$.

As $\theta$ is positive, it implies that
$u^*_{p_n}$ solves $-\Delta u =  \abs{u}^{p_n-1}u$ on $A$.

We conclude by working in the same way as in the first case.

\end{proof}

\subsection{Rescaling argument in $\Omega$ around $x_p^+$: limit
  equation in $\IR^2$}

The idea is inspired by~\cite{grossi}.  Let us consider
$\Omega^{+}(\varepsilon_p):= \frac{\Omega - x_p^+}{\varepsilon_p}$ and
$z_{p}:\Omega^{+}(\varepsilon_p)\to \IR$ the scaling of $u_{p}$ around
$x_{p}^+$:
\begin{equation}
\label{eq13}
z_p(x):=\frac{p}{u_p(x_p^+)}(u_p(\varepsilon_px+ x_p^+)-u_p(x_p^+)).
\end{equation}

\textit{Proof of Theorem~\ref{intro2}} : Let $p_n$ be a sequence,
$p_n\to +\infty$. As in the previous proof, we have that $z_{p_n}$ solves
the equation
\begin{equation*}
\left\{
\begin{aligned}
-\Delta z_{p_n}&= \left\vert
  1+\frac{z_{p_n}}{p_n}\right\vert^{p_n-1}\left(1+\frac{z_{p_n}}{p_n}\right),
&\text{  in } \Omega^+ (\varepsilon_{p_n}),  \\
&\left\vert 1+\frac{z_{p_n}}{p_n}\right\vert\leq 1& \\
z_{p_n}&=-p_n,&\text{ on } \partial \Omega^+ (\varepsilon_{p_n}).
\end{aligned}
\right.
\end{equation*}

Let us fix $R>0$. By Proposition~\ref{away...}, we know that
$\frac{d(x_{p_n}^{+}, \partial
  \Omega)}{\varepsilon_{p_n}}\to +\infty$. So, $\Omega^+ (\varepsilon_{p_n})$
``converges'' to $\IR^{2}$ as $p_n\to +\infty$, i.e.\  $B(0,R)\subset
\Omega^+(\varepsilon_{p_n})$ for large $n$.  Let us consider the problem
\begin{equation*}
\left\{
\begin{aligned}
-\Delta w_{p_n}&= \left\vert
  1+\frac{z_{p_n}}{p_n}\right\vert^{p_n-1}\bigl(1+\frac{z_{p_n}}{p_n}\bigr),&&\text{  in } B(0,R), \\
w_{p_n}&=0,&&\text{ on } \partial B(0,R).
\end{aligned}
\right.
\end{equation*}
Since, by~\eqref{eq13}, $\left\vert
  1+\frac{z_{p_n}}{p_n}\right\vert\leq 1$, we get  that
$\abs{w_{p_n}}\leq C$ independent of $n$. By arguing as before, we get
that $\psi_{p_n}$ and $z_{p_n}$ are bounded up to a subsequence in
$L^{\infty}(B(0,R))$ for any $R$.

Thus, by the standard regularity theory, $z_{p_n}$ is bounded in
$C^2_{loc}(\IR^2)$ and, on
each ball, $1+\frac{z_{p_n}}{p_n}>0$ for large $n$.   We have that
$z_{p_n}\to z$ in $C^{2}_{\text{loc}}(\IR^2)$ and $-\Delta z = e^z$.

To finish, we prove that $\int_{\IR^2}e^z< +\infty$.  We have that $z_{p_n}+p_n
\bigl(\log\bigl\vert
1+\frac{z_{p_n}}{p_n}\bigr\vert-\frac{z_{p_n}}{p_n})$ converges
pointwisely to $z$ in $\IR^2$. By Fatou's lemma, we deduce
\begin{equation*}
\begin{split}
\int_{\IR^2} e^z &\leq \lim_{n}\int_{\Omega^+
  (\varepsilon_{p_n})}e^{z_{p_n}+p_n\left(\log\left\vert
      1+\frac{z_{p_n}}{p_n}\right\vert-\frac{z_{p_n}}{p_n}\right)}\\
&= \lim_n \int_{\Omega^+ (\varepsilon_{p_n})} \left\vert
  1+\frac{z_{p_n}}{p_n}\right\vert^{p_n}\\
&\leq \lim_{n}\int_{\Omega}\frac{\abs{u_{p_n}}^{p_n}}{\varepsilon_{p_n}^2
  \abs{u_{p_n}(x_{p_n}^+)}^{p_n}}\\
&=\lim_n \int_{\Omega}
\frac{p_n}{\abs{u_{p_n}(x_{p_n}^+)}}\abs{u_{p_n}}^{p_n}\\
&\leq \lim_n
\frac{p_n}{\abs{u_{p_n}(x_{p_n}^+)}}\abs{\Omega}^{\frac{1}{p_n+1}}(\int_{\Omega}\abs{u_{p_n}}^{p_n+1})^{p_n/(p_n+1)}.
\end{split}
\end{equation*}
By Proposition~\ref{estinfl} and Remark~\ref{rem}, we deduce that
$\int_{\IR^2} e^z \leq 16\pi e$.   The solutions of $-\Delta z = e^z$ with
$\int_{\IR^2}e^z< +\infty$ are given by $z(x)= \log\left(\frac{\mu}{(1+
  \frac{\mu}{8}\abs{x-x_0}^2)^2}\right)$ for some $\mu >0$.

As $z(x)\leq z(0)=0$ for any $x$, we have that $\mu =1$ and
$x_0=0$. Finally,
$\int_{\IR^2}e^z=8\pi.$ \begin{flushright}$\square$\end{flushright}

\subsection{Rescaling argument in the positive nodal domain}

Theorem~\ref{intro2} implies directly a control on
$d(x_p^{+}, NL_p)$ where $NL_p$ denotes the nodal line of
$u_p$.
\begin{Prop}
\label{controlNL}
We have
\begin{equation}
\label{eq15} \frac{d(x_p^{+},   NL_{p_n})}{\varepsilon_p}\to +\infty
\end{equation}
as $p\to +\infty$.
\end{Prop}
\begin{proof}
If the assertion is not true then, for a sequence
$p_n\rightarrow+\infty$ the level curve $C_{p_n}(z_{p_n})=\{ x\in
\Omega^+(\varepsilon_{p_n}), z_{p_n}(x)=-p_n\}$ intersects
$B(0,R)$ for some large $R>0$. This is a contradiction since
$z_{p_n}$ is uniformly bounded in all balls.
\end{proof}

\textit{Proof of Theorem~\ref{intro3}} :  By
Proposition~\ref{controlNL}, we can repeat the proof of
Theorem~\ref{intro2} for the rescaled function $z_p(x)$ in
$\Tilde{\Omega}_p^+$. \begin{flushright}$\square$\end{flushright}

\subsection{Rescaling argument on  $\Omega$ around $x_p^-$}

Let us consider
$\Omega^{-}(\varepsilon_p):= \frac{\Omega - x_p^-}{\varepsilon_p}$ and
$z_{p}^-:\Omega^{-}(\varepsilon_p)\to \IR$ the scaling of $u_{p}$
around $x_{p}^-$:
\begin{equation}
\label{eq13a} z_p^-(x):=\frac{p}{u_p(x_p^+)}(-u_p(\varepsilon_px+
x_p^-)-u_p(x_p^+)).
\end{equation}
To obtain the same kind of result as that of Theorem~\ref{intro2}, we need
\begin{equation}
\label{refbla} \frac{d(x_{p}^{-}, \partial
\Omega)}{\varepsilon_p}\to +\infty
\end{equation}
 as  $p\to +\infty$. To get~\eqref{refbla} we can repeat step by step the
proof of Proposition~\ref{away...}. The only delicate point is the
use of Harnack's inequality when we need that $\psi_p(0)$ is
bounded from below. Nevertheless, requiring that $p (u_p(x_p^+)+
u_p(x_p^-))$ is bounded (alternative ($B$) in the introduction) we
get the boundness of $\psi_p(0)$ and so ~\eqref{refbla} holds.
This explains the role of condition ($B$) in getting
Theorem~\ref{intro2bis}.

\textit{Proof of Theorem~\ref{intro2bis}} : it is obtained
following step by step the proof of Theorem~\ref{intro2}. The
constant $\mu$ in the limit function $z$ can be different from $1$
because
$$z(0)=\lim\limits_{p\rightarrow+\infty}z_p^-(0)=
\lim\limits_{p\rightarrow+\infty}\frac{p}{u_p(x_p^+)}\left(-u_p(x_p^-)-u_p(x_p^+)
\right)\ne0$$ whenever $K$ (in condition ($B$)) is not zero.
\begin{flushright}$\square$\end{flushright}

\subsection{Rescaling argument in the negative nodal domain}

We would like to obtain a result similar to that of Theorem~\ref{intro3}
 for the function $u_p^-$ defined in the negative nodal domain
 $\Tilde{\Omega}_p^-$. We consider solutions satisfying condition $(B)$. By
Theorem~\ref{intro2bis},  working in the same way as in the proof of
Proposition~\ref{controlNL}, we get that
$\varepsilon_{p}^{-1}d(x_{p}^-,NL_{p})\to +\infty$.

\textit{Proof of Theorem~\ref{intro4}} : As ~\eqref{refbla} is
satisfied when $(B)$ holds, we can repeat
 the proof of Theorem~\ref{intro3}, taking into account the remark
 in the proof of Theorem~\ref{intro2bis}.
\begin{flushright}$\square$\end{flushright}
We conclude this section by explaining  the Remark~\ref{R1} of the
introduction. Let us now consider the "natural" rescaling
coefficient
\begin{equation*}
\tilde\varepsilon_{p}^2 :=\frac{1}{p\abs{u_p^-(x_p^-)}^{p-1}} \to 0
\end{equation*}
since  $\liminf_{p\to +\infty} \abs{u_p(x_p^-)}^{p-1}\geq 1$
(see Remark~\ref{rem2}).  We would like to control  the rescaling of
$u_p^-$ around $x_p^-$
\begin{equation*}
z_p(x):= \frac{p}{u_p(x_p^-)}\left( u_p^-(\tilde\varepsilon_{p}
  x+x_p^-)-u_p(x_p^-)\right).
\end{equation*}

The same argument as in the proof of Proposition~\ref{away...} and
Theorem~\ref{intro2} does not work as we might loose the essential
estimate $\abs{1+ \frac{z_p}p}\leq 1$ in the proof. So, we do not
get Proposition~\ref{controlNL} for $x_p^-$ and we need to assume
that $\tilde\varepsilon_{p}^{-1} d(x_p^-, NL_p) \to +\infty$.

\begin{Prop}
\label{away2...} Assume that $\tilde\varepsilon_{p}^{-1}d(x_{p}^-,
NL_{p})\to +\infty$ as $p\rightarrow+\infty$, with
$\tilde\varepsilon_{p}$ defined as $\tilde\varepsilon_{p}^2
:=\frac{1}{{p}\abs{u_{p}^-(x_{p}^-)}^{{p}-1}}$. Then $
\tilde\varepsilon_{p}^{-1}d(x_{p}^{-}, \partial
  \Omega)\rightarrow+\infty$ as $p\to +\infty$.
\end{Prop}

\begin{proof}

Let us work by contradiction and assume that, for a sequence
$p_n\rightarrow+\infty$, $\frac{d(x_{p_n}^{-}, \partial
\Omega)}{\tilde\varepsilon_{p_n}}\to l\geq 0$.  Let us also assume
w.l.o.g.\ that $x_{p_n}^-\to x_* \in\partial \Omega$ (i.e.\
$\frac{d(x_{p_n}^{-}, x_*)}{\tilde\varepsilon_{p_n}}\to l$).

As  $\tilde\varepsilon_{p_n}^{-1}d(x_{p_n}^-,
NL_{p_n})\rightarrow+\infty$, we can construct a sequence  of
$\C^1$-domains $D_{p_n}$ which are the intersection between a
neighborhood $V_{p_n}$ of $x_*$ and $\Tilde{\Omega}_{p_n}^-$ such that
\begin{equation*}
\tilde\varepsilon_{p_n}^{-1}d(x_*, \partial D_{p_n}
\setminus \partial\Omega)\to + \infty.
\end{equation*}
For large $n$, we have that $x_{p_n}^-$  belongs to $D_{p_n}$.
So, as $u_{p_n}$ stays negative in $D_{p_n}$, we can argue in the same
way as Proposition~\ref{away...} to conclude the proof.
\end{proof}

By working in the same way as in Theorem~\ref{intro2} or Theorem~\ref{intro3},
Proposition~\ref{away2...} allows to make
the rescaling  in the negative nodal domain $\Tilde{\Omega}_{p}^-$, so to
obtain for $\frac{p}{u_p(x_p^-)} (u_p^-(\tilde\varepsilon_p x +
x_p^-) -u_p(x_p^-))$ the same assertion as for $z_p$ in
Theorem~\ref{intro3}.

\section{$L^\infty$-estimates}

In the last two sections, we will work in the positive and
negative nodal domains. While dealing with  the positive nodal
domain, $z_p$ will always denote the rescaled function used in
Theorem~\ref{intro2}. For the negative one, the expression of
$z_{p}$ can be defined as in Theorem~\ref{intro4} when $(B)$ holds
(and so with $\varepsilon_p^{-2}= p \abs{u_p(x_p^+)}^{p-1}$).

Let us point out that some proofs will be given just for the
positive case, the negative one being similar.

  \begin{Prop}
  \label{goodb}
For any sequence $p_n\rightarrow+\infty$ we have $\limsup_{n\to
+\infty}
  \abs{u_{p_n}^\pm(x_{p_n}^\pm)}\leq e^{1/2}$.
  \end{Prop}
 \begin{proof}
Let us prove the assertion for the positive case.  By Fatou's
lemma, we have
 \begin{equation*}
\begin{split}
1&=\frac{\int_{\Omega}\abs{u_{p_n}^+}^{p_n+1}}{\norm{u_{p_n}^+}_{p_n+1}^{p_n+1}}=\left(\frac{\abs{u_{p_n}^+(x_{p_n}^+)}}
 {\norm{u_{p_n}^+}_{p_n+1}}\right)^{p_n+1}\varepsilon_{p_n}^2
\int_{\Tilde{\Omega}^{+}(\varepsilon_{p_n})}\left\vert1+\frac{z_{p_n}}{p_n}\right\vert^{p_n+1}\\
&= \frac{\abs{u_{p_n}^+(x_{p_n}^+)}^2}{p_n\norm{u_{p_n}^+}_{p_n+1}^{p_n+1}}\int_{\Tilde{\Omega}^{+}(\varepsilon_{p_n})}\left\vert1+\frac{z_{p_n}}{p_n}\right\vert^{p_n+1}\\
& \geq
\frac{\limsup_{n\to +\infty} \abs{u_{p_n}^+(x_{p_n}^+)^2}}{8\pi e}\int_{\IR^2}e^z.
\end{split}
\end{equation*}
As $\int_{\IR^2}e^z=8\pi$, the proof is complete.
\end{proof}

Now, we study the equality in the last statement. We will show that
$\int_{\Tilde{\Omega}^{\pm}(\varepsilon_{p_n})}\left\vert1+\frac{z_{p_n}}{p_n}\right\vert^{p_n+1}$
converges to $\int_{\IR^2}e^z$ with no mass lost at infinity.

Let us consider the linearized operators
\begin{equation*}
L_{p}^+(v) = -\Delta v -p \abs{u_p^+}^{p-1} v
\end{equation*}
for $v:\Tilde{\Omega}_p^+\to \IR$ and let us denote by
$\lambda_i(L^+_p)$ the eigenvalues
of $L^+_p$ with homogenous Dirichlet boundary conditions.

Our aim is to prove Theorem \ref{introthm}, therefore we assume
that the Morse index of $u_p^+$ in $\Omega^+$ is $1$. Hence we
have
$$\lambda_1(L^+_p) < 0\quad\hbox{ and}\quad \lambda_2(L^+_p)\geq 0
\hbox{ in }\Tilde{\Omega}_{p}^+.$$ Then, for $D\subset
\Tilde{\Omega}^+(\varepsilon_{p})$, let us consider
$L^+_{p,D}(v)=-\Delta v -\frac{\abs{u_{p}^+(\varepsilon_{p} x +
    x_{p}^+)}^{p-1}}{\abs{u_{p}^+(x_{p}^+)}^{p-1}}v$
and denote by $\lambda_i(L^{+}_{p,D})$ the corresponding Dirichlet
eigenvalues. By scaling, we get
\begin{Lem}
\label{lem111}
$\lambda_1(L^{+}_{p,\Tilde{\Omega}^{+}(\varepsilon_{p})}) < 0$ and
  $\lambda_2(L^{+}_{p,\Tilde{\Omega}^{+}(\varepsilon_{p})}) \geq 0$.
\end{Lem}

\begin{Lem}
\label{careful} Let $p\to +\infty$, there exists $r>0$ such that
$\lambda_1(L^{+}_{p,B(0,r)})<0$ for large $p$.
\end{Lem}
\begin{proof}
Let us consider $w_{p}= x.\nabla z_{p} + \frac{2}{p-1}z_{p} +
\frac{2p}{p-1}$. We have that $w_{p}$ satisfies $-\Delta w =
\frac{\abs{u_{p}^+(\varepsilon_{p}x +
 x_{p}^+)}^{p-1}}{\abs{u_{p}^+(x_{p}^+)}^{p-1}}w$.

We also have $w_{p}(0)\to 2$. As $z_{p} \to
z=\log\big(\frac{1}{\bigl(1+\frac{\abs{x}^2}{8}\bigr)^2}\big)$,
for $\abs{x}=r$, we get
\begin{equation*}
w_{p}(x) \to -\frac{4 r^2}{8 + r^2}+2
\end{equation*}
as $p\to +\infty$. So, for large $r$, $w_{p}\to \alpha <0$  on
$\partial B[0,r]$.

Let us fix such a $r$. By considering $A_{p}:=\{x\in B(0,r):
w_{p}>0\}$ and the function $\Bar{w}_{p}$ equals to $w_{p}$ on
$A_{p}$ ($0$ otherwise), we get

\begin{equation*}
\int_{B(0,r)} \abs{\nabla \Bar{w}_{p}}^2 -\int_{B(0,r)}
\frac{\abs{u_{p}^+(\varepsilon_{p} x +
    x_{p}^+)}^{p-1}}{\abs{u_{p}^+(x_{p}^+)}^{p-1}}\Bar{w}_{p}^2=0,
\end{equation*}
which implies our statement.
\end{proof}

\begin{Lem}
\label{principle1} For $p$ large, $\lambda_1(L^{+}_{p,
  \Tilde{\Omega}^{+}(\varepsilon_{p})\setminus B(0,r)})>0$, where $r$
is given by Lemma~\ref{careful}.
\end{Lem}
\begin{proof}
If $\lambda_1(L^{+}_{p,
  \Tilde{\Omega}^{+}(\varepsilon_{p})\setminus B(0,r)})$ was
negative then, by Lemma~\ref{careful} we would have
$\lambda_2(L^{+}_{p,\Tilde{\Omega}^{+}(\varepsilon_{p})})<0$ which
contradicts Lemma~\ref{lem111}.
\end{proof}

\textit{Proof of Theorem~\ref{intro5} : } Since we are analyzing
nodal solutions which satisfy condition ($B$), it is enough to
prove that $\norm{u_{p}^+}_{\infty}\to \sqrt{e}$ as
$p\rightarrow+\infty$. Let us argue by contradiction and assume
that for a sequence $p_n\rightarrow+\infty$, by
Proposition~\ref{goodb},  $\lim\limits_{n\to
\infty}\abs{u_{p_n}^+(x_{p_n}^+)}= \lim\limits_{n\to
\infty}\norm{u_{p_n}^+}_{\infty}< e^{1/2}$. We claim that this
implies $z_{p_n}(x)-z(x) \leq C$ on
$\Tilde{\Omega}^{+}(\varepsilon_{p_n})$ uniformly.

Indeed, $z_{p_n}$ converges to $z$ on each compact set. In particular, on
$B(0,r)$, where $r$ is given by Lemma~\ref{careful}. So, it is enough
to check what happens in $\Tilde{\Omega}^{+}(\varepsilon_{p_n})\setminus B(0,r)$.

On one hand, $-\Delta z = e^z \geq \abs{1+\frac{z}{p}}^{p}$ for any
$p>1$. On the other hand, by computing $z_{p_n} -z$ on $\partial
\Tilde{\Omega}^{+}(\varepsilon_{p_n})\setminus B[0,r]$, we get for
some uniform constant $C$
\begin{equation*}
\begin{split}
z_{p_n}(x)-z(x) &= -p_n - \log\left(\frac{1}{(1+ \frac{\abs{x}^2}{8})^2}\right)\\
&
\leq-p_n-\log\left(C\frac{\varepsilon_{p_n}^4}{\hbox{d}(x^+_{p_n},\partial\Omega)^4}\right)\\
& \leq -p_n + 2 \log (p_n
\abs{u_{p_n}^+(x_{p_n}^+)}^{p_n-1})+4\log\left(\hbox{d}(x^+_{p_n},\partial\Omega)\right)+C\\
& \leq -p_n + 2 \log (p_n \abs{u_{p_n}^+(x_{p_n}^+)}^{p_n-1})+C\\
&\leq  -p_n + 2\log (\abs{u_{p_n}^+(x_{p_n}^+)}^{p_n-1}) + C \leq C
\end{split}
\end{equation*}
where we used that  $\abs{u_{p_n}^+(x_{p_n}^+)} < e^{1/2}$ and
$\hbox{d}(x^+_{p_n},\partial\Omega)\leq C$
for large $n$ (by contradiction).

We also have the estimate on $\partial B(0,r)$ because we have the
convergence on each compact set.

Finally, by convexity, we have
\begin{equation*}
-\Delta z_{p_n} +\Delta z \leq \left\vert1+\frac{z_{p_n}}{p_n}\right\vert^{p_n-1}
(z_{p_n} -z) = \frac{\abs{u_{p_n}^+(\varepsilon_{p_n}x+
  x_{p_n}^+)}^{p_n-1}}{\abs{u_{p_n}(x_{p_n}^+)}^{p_n-1}}(z_{p_n}-z).
\end{equation*}
Since the maximum principle holds in
$\Tilde{\Omega}^{+}(\varepsilon_{p_n})\setminus B(0,r)$ for
$L^{+}_{p_n, \Tilde{\Omega}^{+}(\varepsilon_{p_n}) \setminus B(0,r)}$ (see
Lemma~\ref{principle1}), we deduce our claim.

From this claim, we obtain that
$\int_{\Tilde{\Omega}^{+}(\varepsilon_{p_n})} \left\vert
  1+\frac{z_{p_n}}{p_n}\right\vert^{p_n+1}$
converges to $\int_{\IR^2} e^z$. So,
\begin{equation*}
\begin{split}
1&= \frac{\int_{\Omega} \abs{u_{p_n}^+}^{p_n+1}}{\norm{u_{p_n}^{+}}_{p_n+1}^{p_n+1}}=
\frac{\abs{u_{p_n}^+(x_{p_n}^+)}^{p_n+1}}{\norm{u_{p_n}^{+}}_{p_n+1}^{p_n+1}}\varepsilon_{p_n}^2\int_{\Tilde{\Omega}^{+}(\varepsilon_{p_n})}\left\vert
  1+\frac{z_{p_n}}{p_n}
\right\vert^{p_n+1}\\
&= \frac{\norm{u_{p_n}^+}_{\infty}^{2}}{8\pi e + o(1)} (8\pi + o(1)),
\end{split}
\end{equation*}
which proves that $\lim_{n\to\infty}\norm{u_{p_n}^+}_{\infty}=e^{1/2}$,
which is a contradiction.
\begin{flushright}$\square$\end{flushright}

\section{Green's characterizations}

To start with, we observe that  Theorem~\ref{intro5} gives a
direct way to prove  the convergence of
$\int_{\Tilde{\Omega}^{\pm}(\varepsilon_{p})} \left\vert
1+\frac{z_{p}}{p}\right\vert^{p+1}$, as $p\to +\infty$.

\begin{Prop}
\label{conv1} As $p\to +\infty$
$\int_{\Tilde{\Omega}^{\pm}(\varepsilon_{p})}\left\vert
  1+\frac{z_{p}}{p}\right\vert^{p+1}\rightarrow\int_{\IR^2}e^z=8\pi$.
\end{Prop}

\begin{proof}
Let us give the proof for the positive case.  For any $n\in\IN$, we have
\begin{equation*}
\begin{split}
\int_{\Tilde{\Omega}^{+}(\varepsilon_{p})}\left\vert
  1+\frac{z_{p}}{p}\right\vert^{p+1} &=\frac{p
  \int_{\Omega}\abs{u_{p}^+}^{p+1}}{\abs{u_{p}^+(x_{p}^+)}^2}.
\end{split}
\end{equation*}
As the right-hand side converges to $8\pi$, we obtain our statement.
\end{proof}

The previous result implies a similar statement where the exponent
$p+1$ is replaced by $p$.
\begin{Prop}
\label{conv2}
We have
\begin{equation*}
 \int_{\Tilde{\Omega}^{\pm}(\varepsilon_{p_n})}\left\vert
  1+\frac{z_{p}}{p}\right\vert^{p}\to 8\pi
\end{equation*}
as $p\to +\infty$.

\end{Prop}

\begin{proof}
Let us give the proof for the positive case.  On one hand, as $\left\vert
  1+\frac{z_{p}}{p}\right\vert\leq 1$, we have
$\int_{\Tilde{\Omega}^{+}(\varepsilon_{p})}\left\vert
  1+\frac{z_{p}}{p}\right\vert^{p+1}\leq
\int_{\Tilde{\Omega}^{+}(\varepsilon_{p})}\left\vert
  1+\frac{z_{p}}{p}\right\vert^{p}$. By
Proposition~\ref{conv1}, we get $8\pi \leq \liminf\limits_{p\to
+\infty} \int_{\Tilde{\Omega}^{+}(\varepsilon_{p})}\left\vert
  1+\frac{z_{p}}{p}\right\vert^{p}$.

On the other hand,  as
$\int_{\Tilde{\Omega}^{+}(\varepsilon_{p})}\left\vert
  1+\frac{z_{p}}{p}\right\vert^{p+1}\to 8\pi$ and $\left\vert
  1+\frac{z_{p}}{p}\right\vert^{p+1}\to e^z$ with
$\int_{\IR^2}e^z=8\pi$ (see Theorem~\ref{intro3} and
Theorem~\ref{intro4}), we have
\begin{equation}
\label{gree} \forall \varepsilon >0,\ \exists R_{\varepsilon}>0
\text{ and }p_\varepsilon : \forall p>p_\varepsilon\ \text{ and }
\ R>R_\varepsilon, \ \int_{\Tilde{\Omega}^{+}(\varepsilon_{p})\cap
\{\abs{x}> R\}}\left\vert
  1+\frac{z_{p}}{p}\right\vert^{p+1} \leq \varepsilon.
\end{equation}
By interpolation, we get for any $\varepsilon>0$ that
\begin{equation*}
\begin{split}
\int_{\Tilde{\Omega}^{+}(\varepsilon_{p})}&\left\vert
  1+\frac{z_{p}}{p}\right\vert^{p} =
\int_{\Tilde{\Omega}^{+}(\varepsilon_{p})\cap\{ \abs{x}\leq
R_{\varepsilon}\}} \left\vert
  1+\frac{z_{p}}{p}\right\vert^{p} +
\int_{\Tilde{\Omega}^{+}(\varepsilon_{p})\cap \{\abs{x}>
R_\varepsilon\}}\left\vert
  1+\frac{z_{p}}{p}\right\vert^{p}    \\
&\leq  \int_{\Tilde{\Omega}^{+}(\varepsilon_{p})\cap
  \{\abs{x}\leq R_{\varepsilon}\}}
\left\vert
  1+\frac{z_{p}}{p}\right\vert^{p} +
\left(\int_{\Tilde{\Omega}^{+}(\varepsilon_{p})\cap
    \{\abs{x}>R_\varepsilon\}}\left\vert
  1+\frac{z_{p}}{p}\right\vert^{p+1}\right)^{\frac{p}{p+1}}
\abs{\Tilde{\Omega}^{+}(\varepsilon_{p})}^{\frac{1}{p+1}}\\
&\leq  \int_{\Tilde{\Omega}^{+}(\varepsilon_{p})\cap
\{\abs{x}\leq R_{\varepsilon}\}} \left\vert
  1+\frac{z_{p}}{p}\right\vert^{p} +
\left(\int_{\Tilde{\Omega}^{+}(\varepsilon_{p})\cap
    \{\abs{x}>R_\varepsilon\}}\left\vert
  1+\frac{z_{p}}{p}\right\vert^{p+1}\right)^{\frac{p}{p+1}}\abs{\Omega}^{\frac{1}{p+1}}\varepsilon_{p}^{\frac{-1}{p+1}}.
\end{split}
\end{equation*}
As $ \int_{\Tilde{\Omega}^{+}(\varepsilon_{p})\cap
  \{\abs{x}\leq R_\varepsilon\}}
\left\vert
  1+\frac{z_{p}}{p}\right\vert^{p}\to C \leq 8\pi$ and
$\varepsilon_{p}^{\frac{-1}{p+1}}\to e^{1/4}$  as $p\to +\infty$,
we get by~\eqref{gree} that, for any $\varepsilon >0$, there
exists $\Bar{p}>0$ such that if $p>\Bar{p}$ then
\begin{equation*}
   \int_{\Tilde{\Omega}^{+}(\varepsilon_{p})}\left\vert
  1+\frac{z_{p}}{p}\right\vert^{p} \leq (8\pi +\varepsilon)
+ \varepsilon^{\frac{p}{p+1}} (e^{\frac{1}{4}}+ \varepsilon),
\end{equation*}
which implies  our statement.
\end{proof}
Let us denote by $G$  the Green's function of $\Omega$ and by
$x^\pm\in \Bar{\Omega}$ the limit points of $x_{p}^\pm$ as $p\to
+\infty$.

\begin{Lem}
\label{lemgreen}
Let $x\not = x^\pm$. We have
\begin{equation*}
\int_{\Tilde{\Omega}^\pm (\varepsilon_{p})}
G(x,\varepsilon_{p}\psi + x_{p}^\pm)\left\vert
  1+\frac{z_{p}}{p}\right\vert^{p}\intd \psi\to 8\pi G(x,x^\pm).
\end{equation*}
\end{Lem}

\begin{proof}
Let us make the proof for the positive case.  Let us fix $x\not =
x^+$ and consider $\alpha>0$ such that $B(x,\alpha)\subset \Omega$
and $d(x^+, B(x,\alpha))=\beta >0$. We have
\begin{equation*}
\begin{split}
\int_{\Tilde{\Omega}^+ (\varepsilon_{p})} G(x,\varepsilon_{p}\psi
+ x_{p}^+)\left\vert
  1+\frac{z_{p}}{p}\right\vert^{p}d\psi =& \int_{\Tilde{\Omega}^+
  (\varepsilon_{p})\setminus \frac{B(x,\alpha)-x_{p}^+}{\varepsilon_{p}}}
G(x,\varepsilon_{p}\psi + x_{p}^+)\left\vert
  1+\frac{z_{p}}{p}\right\vert^{p}d\psi\\
&+   \int_{\frac{B(x,\alpha)-x_{p}^+}{\varepsilon_{p}}}
G(x,\varepsilon_{p}\psi + x_{p}^+)\left\vert
  1+\frac{z_{p}}{p}\right\vert^{p}d\psi.
\end{split}
\end{equation*}
Arguing as in Proposition~\ref{conv2}, since
$G(x,\varepsilon_{p}\psi + x_{p}^+)$ converges uniformly to $G(x,
x^+)$ on each compact set of $\IR^2$, $G(x,\cdot)$ is bounded on
$\Omega \setminus B(x,\alpha)$ and
$d\left(\frac{B(x,\alpha)-x_{p}^+}{\varepsilon_{p}},0\right)\rightarrow+\infty$,
we get that the first integral converges to $8\pi G(x,x^+) $.
Concerning the second integral, since $x^+\not\in B(x,\alpha)$, we
derive that
$\int_{\frac{B(x,\alpha)-x_{p}^+}{\varepsilon_{p}}}\left\vert
  1+\frac{z_{p}}{p}\right\vert^{p}=p\int_{B(x,\alpha)}\left\vert
  u_{p}^+\right\vert^{p}\rightarrow0$. From the last statement we
deduce that we can apply Lemma 3.5 in~\cite{rw} and obtain that
$\frac{u_{p}}{p\int_{B(x,\alpha)}\left\vert
u_{p}^+\right\vert^{p}}$ is bounded in $B(x,\alpha)$ and hence
$u_{p}(x)<\frac12$ in $B(x,\alpha)$. Then
\begin{equation*}
\begin{split}
&\int_{\frac{B(x,\alpha)-x_{p}^+}{\varepsilon_{p}}}
G(x,\varepsilon_{p}\psi + x_{p}^+)\left\vert
  1+\frac{z_{p}}{p}\right\vert^{p}=p\int_{B(x,\alpha)}G(x,y)\left\vert
  u_{p}^+(y)\right\vert^{p}dy\\
  &\leq p \left(\frac12\right)^{p}\int_{B(x,\alpha)}G(x,y)=o(1),
\end{split}
\end{equation*}
which gives our claim.
\end{proof}

Let us remark that the convergence in Lemma~\ref{lemgreen} is
uniform in $x$ in $\C^0_{\text{loc}}(\Omega\setminus \{x^+\})$.

\begin{Prop}
\label{green} Under the same assumptions as in
Theorem~\ref{intro5}, the following alternatives hold:
\begin{enumerate}
\item $d(x_{p}^+,\partial\Omega)\to 0$ and
  $d(x_{p}^-,\partial\Omega)\not\to 0$. Then the function
  $pu_{p}$ converges,  up to a
  subsequence, to the negative function $-8\pi e^{1/2}G(\cdot , x^-)$
  in $\C^{1}_{\text{loc}}(\Bar{\Omega}\setminus\{x^-\})$ ;
\item $d(x_{p}^-,\partial\Omega)\to 0$ and
  $d(x_{p}^+,\partial\Omega)\not\to 0$. Then the function
  $pu_{p}$ converges,  up to a
  subsequence, to  the positive function $8\pi e^{1/2}G(\cdot, x^+)$
   in $\C^{1}_{\text{loc}}(\Bar{\Omega}\setminus\{x^+\})$ ;
\item $d(x_{p}^+,\partial\Omega)$ and
  $d(x_{p}^-,\partial\Omega)\not\to 0$. Then $p u_{p}$ converges,
  up to a subsequence,  to
  $8\pi e^{1/2}( G(\cdot, x^+)-G(\cdot, x^-))$  in
  $\C^{1}_{\text{loc}}(\Bar{\Omega}\setminus\{x^-, x^+\})$ with $x^+
  \neq x^-,\ x^+,x^- \in\Omega$   ;
\item $d(x_{p}^+,\partial\Omega) \to 0$ and
  $d(x_{p}^-,\partial\Omega)\to 0$. Then $pu_{p}\to 0$  in
  $\C^{1}_{\text{loc}}(\Bar{\Omega}\setminus\{x^-,x^+\})$.
\end{enumerate}
In the case $(3)$, the limit points $x^+$ and $x^-$ satisfy the system
\begin{equation}
\label{3}
\left\{
\begin{aligned}
\frac{\partial G}{\partial x_i}(x^+, x^-)-\frac{\partial H}{\partial
  x_i}(x^+, x^+)=0,\\
\frac{\partial G}{\partial x_i}(x^-, x^+)-\frac{\partial H}{\partial
  x_i}(x^-, x^-)=0,
\end{aligned}
\right.
\end{equation}
for $i=1,2$, where, as in the introduction, $H(x,y)$ is the
regular part of the Green function. Moreover the nodal line of
$u_{p}$  intersects the boundary $\partial \Omega$ for $p$ large.
\end{Prop}

\begin{proof}
We have
\begin{equation*}
\begin{split}
u_{p}(x) &= \int_{\Omega}G(x,y) \abs{u_{p}(y)}^{p-1} u_{p}(y)
\intd y\\
&= \int_{\Tilde{\Omega}_{p}^+}G(x,y) \abs{u_{p}(y)}^{p} \intd y -
\int_{\Tilde{\Omega}_{p}^-}G(x,y) \abs{u_{p}(y)}^{p} \intd y.
\end{split}
\end{equation*}
Let us just treat the first member of the sum. The second one can
be treated in the same way. With the change of variables $y=
\varepsilon_{p}\psi + x_{p}^+$, we get
\begin{equation*}
\begin{split}
\int_{\Tilde{\Omega}_{p}^+}G(x,y) \abs{u_{p}(y)}^{p}\intd y &=
\int_{\Tilde{\Omega}^{+}(\varepsilon_{p})} \frac{1}{p
\abs{u_{p}^+(x_{p}^+)}^{p-1}} G(x, \varepsilon_{p}\psi +
x_{p}^+) \abs{u_{p}^+(\varepsilon_{p}\psi+ x_{p}^+)}^{p} \intd \psi\\
&= \int_{\Tilde{\Omega}^{+}(\varepsilon_{p})} G(x,
\varepsilon_{p}\psi + x_{p}^+)
\frac{\left\vert\abs{u_{p}^+(\varepsilon_{p}\psi+
    x_{p}^+)}-\norm{u_{p}^+}_{\infty}+\norm{u_{p}^+}_{\infty}\right\vert^{p}}{p
  \norm{u_{p}^+}_{\infty}^{p-1}}\intd \psi\\
&=  \int_{\Tilde{\Omega}^{+}(\varepsilon_{p})} G(x,
\varepsilon_{p}\psi + x_{p}^+)
\frac{\left\vert\frac{\norm{u_{p}^+}_{\infty}
      z_{p}}{p}+\norm{u_{p}^+}_{\infty}    \right\vert^{p}}{p
  \norm{u_{p}^+}_{\infty}^{p-1}}\intd \psi\\
&=
\frac{\norm{u_{p}^+}_{\infty}}{p}\int_{\Tilde{\Omega}^{+}(\varepsilon_{p})}
G(x, \varepsilon_{p}\psi + x_{p}^+)\left\vert
  1+\frac{z_{p}}{p}\right\vert^{p} \intd\psi.
\end{split}
\end{equation*}
As $\norm{u_{p}^+}_{\infty}\to e^{1/2}$ and
$\int_{\Tilde{\Omega}^{+}(\varepsilon_{p})} G(x,
\varepsilon_{p}\psi + x_{p}^+)\left\vert
  1+\frac{z_{p}}{p}\right\vert^{p} \intd\psi$ converges to
$8\pi G(x,x^+)$ (see Lemma~\ref{lemgreen}),
by working in the same way with the second part of the sum, we have
\begin{equation}
\label{1} pu_{p} \to 8\pi e^{1/2} (G(., x^+)-G(.,x^-))
\end{equation}
in $\C^{0}_{\text{loc}}(\Omega\setminus \{x^+, x^-\})$,  up to a
subsequence. By regularity, it implies the convergence in
$\C^1_{\text{loc}}(\Omega\setminus \{x^+\})$ (see~\cite{han}).

Observing that $G(.,x^+)=0$ when $x^+\in \partial\Omega$, we get
the alternatives. In the third case, we prove that $x^+\neq x^-$
as follows. Indeed, arguing by contradiction, we have that
$x^+=x^-$. Then, $p u_{p}\to 0$ in $\C^1(\Bar{\omega})$ where
$\omega$ is a neighborhood of the boundary $\partial\Omega$. By
the Pohozaev identity, multiplying by $p^2$, we get
\begin{equation*}
\frac{p^2}{p+1} \int_{\Omega}\abs{u_{p}}^{p+1} = \frac{1}{4}
\int_{\partial \Omega} (x\cdot \nu) (\partial_\nu(pu_{p}))^2.
\end{equation*}
As the left-hand side converges to $16\pi e$ (see
Remark~\ref{rem}) and the right-hand side converges to $0$ (since
$p u_{p}\to 0$ in $\C^1(\Bar{\omega})$), we get a contradiction.

Now, we prove that $x^+$ and $x^-$ solve the system \eqref{3}.  Concerning
the location of $x^+$ and $x^-$, we use a Pohozaev-type
identity. For $i=1,2$ let us multiply the equation ~\eqref{pblP} by
$\frac{\partial u_p}{\partial x_i}$ and
integrate on $B_R(x^+)\subset\Omega$, the ball centered at $x^+$ and
radius $R$. We have that,
\begin{eqnarray}
\label{robin1} &&0=\frac2{p+1}\int_{\partial
  B_R(x^+)}|u_{p}|^{p+1}\nu_i+\int_{\partial
  B_R(x^+)}\frac{\partial u_{p}}{\partial x_i}
\frac{\partial u_{p}}{\partial\nu}-\frac12\int_{\partial
  B_R(x^+)}|\nabla u_{p}|^2\nu_i=\nonumber\\
&&I_1+I_2+I_3
\end{eqnarray}
where $\nu_i$ are the components of the normal direction.\\
From ~\eqref{1} we get that
\begin{equation}
\label{2} p^2I_1=O\left(\frac12\right)^{p}\quad\hbox{as
}p\rightarrow +\infty.
\end{equation}
Multiplying~\eqref{robin1} by $p^2$ and using~\eqref{1} and
~\eqref{2} we deduce
\begin{equation}
\label{21} \int_{\partial B_R(x^+)}\frac{\partial (G(.,
x^+)-G(.,x^-))}{\partial x_i} \frac{\partial (G(.,
x^+)-G(.,x^-))}{\partial\nu}- \frac12\int_{\partial
B_R(x^+)}|\nabla (G(., x^+)-G(.,x^-))|^2\nu_i=0.
\end{equation}
The last integral was computed in ~\cite{MaWei}, page 511-512 which gives
\begin{equation}
\nabla\left(G(x^+,x^-)-H(x^+,x^+)\right)=0
\end{equation}
Repeating the same procedure in $B_R(x^-)$ we derive that
\begin{equation}
\nabla\left(G(x^-,x^+)-H(x^-,x^-)\right)=0
\end{equation}
which gives the claim.

To conclude the proof, we show that the nodal line of $u_{p}$
intersects the boundary $\partial\Omega$ for $p$ large. If not,
$u_{p}$ is a one-signed function in a neighborhood  of
$\partial\Omega$, which, by Höpf's lemma, implies that
$\partial_\nu pu_{p}$ is one-signed on $\partial\Omega$ for large
$p$. On the other hand, as $x^+ \neq x^-$ and
$\int_{\partial\Omega} \partial_\nu (G(\cdot, x^+)-G(\cdot,
x^-))=0$, the normal derivative of the limit function changes its
sign along $\partial\Omega$. It contradicts the $\C^1$-convergence
of $pu_{p}$ to $8\pi \sqrt{e} (G(\cdot, x^+)-G(\cdot, x^-))$ in a
compact neighborhood of $\partial\Omega$.
\end{proof}

\textit{Proof of Theorem~\ref{introthm}} : We need to prove that the
cases $(1)$, $(2)$ and $(4)$ in Proposition~\ref{green} cannot happen. To
start with, we focus on the case $(4)$. Arguing by contradiction, let us
assume that $x^+$ and $x^-$
  belong to $\partial\Omega$. Let $D\subset \Omega$ be an open domain
  which is the intersection between a neighborhood of  $x^+$ and
  $\Omega$. We assume w.l.o.g.\ that $x^-\notin \Bar{D}$  when $x^+
   \neq x^-$ and $x^-\notin \partial\Bar{D}$ when $x^+ = x^-$. We have that
  $pu_{p}\to 0$ in
  $\C^1_\text{loc}(\Bar{D}\setminus\{x^+\})$. Using the same
  notations as in the proof of Proposition~\ref{away...} (for $Q$, $Q^+$,
  $S$,...), we consider the change of
variables $\phi : D \to Q^+$ and $\phi(D\cap\partial\Omega )=S $.
Moreover $\phi^{-1} \in \C^1$. Then, we define $D^*:=
\phi^{-1}(Q)$ and $u_{p}^*$ which is $u_{p}$ on $D$ and the odd
tubular reflection on $D^*\setminus D$ (as in the proof of
Proposition~\ref{away...}). We get that $u_{p}^*$ solves $-\Delta
u = \abs{u}^{p-1}u$ on $D^*$ and $p u_{p}^*\to 0$ in
$\C^1(\Bar{\omega}^*)$ where $\omega^*\subset D^*$ is a
neighborhood of the boundary $\partial D^*$ avoiding $x^+$. Using
the Pohozaev identity and  multiplying by $p^2$, we get the
existence of constants $K,K^*$ and $K^{**}$ such that
\begin{equation}
\label{poho}
\begin{aligned}
\frac{p^2}{p+1} \int_{D^*}\abs{u_{p}^*}^{p+1} = K \int_{\partial
D^*} (x\cdot \nu) (\partial_\nu(pu_{p}^*))^2\intd\tau &+
K^*\int_{\partial D^*} (x\cdot \nu) (\partial_{\tau}(pu_{p}^*))^2
\intd\tau +\\
K^{**}\frac{p^2}{p+1}  \int_{\partial D^*} \abs{u_{p}^*}^{p+1}.
\end{aligned}
\end{equation}
As $p u_{p}^*\to 0$ in $\C^1(\Bar{D^*}\setminus\{x^+\})$, the
right-hand side is converging to zero. To get a contradiction, we
prove that the left-hand side is not converging to zero. For this,
we claim that $p\int_{D^*} \abs{u_{p}^*}^{p+1}$ converges to
$C\geq 8\pi e$.  If not, as $p\int_{\Omega} \abs{u_{p}^-}^{p+1}\to
8\pi e $ and $p\int_{\Omega} \abs{u_{p}}^{p+1}\to 16\pi e $, we
get the existence of a positive constant $\psi$ such that
\begin{equation*}
\int_{\Omega\setminus (D^*\cup B(x^-,\delta))} p
\abs{u_{p}}\abs{u_{p}}^{p}  >\psi
\end{equation*}
for any $\delta >0$ and large $p$. It contradicts $pu_{p}\to 0$ in
$\C^1(\Bar{\Omega}\setminus\{x^+\})$.

To finish, let us prove that the case $(1)$ cannot happen (the
case $(2)$ is similar). Working in the same way, we  construct an
open domain $x^+\in D^*$ such that $u_{p}^*$ solves $-\Delta u =
\abs{u}^{p-1}u$ on $D^*$ and $p u_{p}^*\to G(\cdot,x^-)$ in
$\C^1(\Bar{\omega}^*)$ where $\omega^*$ is any compact set in
$\Bar{D^*}\setminus\{x^+\}$. Using again the Pohozaev identity and
multiplying by $p^2$, we get equation~\eqref{poho}. Working as
previously, as $p u_{p}\to G(.,x^-)$ and $u_{p}\to 0$ in
$\C^1_{\text{loc}}(\Omega\setminus\{x^+\})$, the left-hand side
converges to $C\geq 8\pi e$. Concerning the right-hand side, as
$G(.,x^-)\in \C^1(\Bar{\Omega})$ and $G(\cdot,x^-)  =0$ on
$\partial\Omega$,  we can consider $D^*$ small enough such that
the two last terms converge to constants less than $8\pi e /3$.
For the first one, as there exists a constant $K>0$ such that
$\abs{\nabla
  G(x,y)}\leq \frac{K}{\abs{x-y}}$, we get
that   $(\partial_\nu G(\cdot,x^-))^2$ is bounded in a neighborhood of
$x^+$. Taking $D^*$ small enough, we also get that the first term
converges to a constant less than $8\pi e/3$ which is
a contradiction. \begin{flushright}$\square$\end{flushright}

\bibliographystyle{plain}
\bibliography{paperGGP}

\end{document}